\documentclass[12pt]{amsart}
\usepackage{amsmath}
\usepackage{amssymb}
\usepackage{amsthm}
\usepackage{enumerate}
\usepackage[mathscr]{eucal}
\usepackage{color}
\usepackage{graphics}
\theoremstyle{plain}
\newtheorem{theorem}{Theorem}[section]

\newtheorem{cor}{Corollary}[theorem]

\theoremstyle{definition}
\newtheorem{definition}{Definition}[section]
\newtheorem{remark}{Remark}[section]

\usepackage[pagewise]{lineno}
\begin{document}
\title[Approximate Birkhoff-James orthogonality ]{ Approximate Birkhoff-James orthogonality in the space of  bounded linear operators}
\author[ Kallol Paul, Debmalya Sain and  Arpita Mal ]{ Kallol Paul, Debmalya Sain and Arpita Mal}

\newcommand{\acr}{\newline\indent}

\address[Paul]{Department of Mathematics\\ Jadavpur University\\ Kolkata 700032\\ West Bengal\\ INDIA}
\email{kalloldada@gmail.com}

\address[Sain]{Department of Mathematics\\ Indian Institute of Science\\ Bengaluru 560012\\ Karnataka \\India\\ }
\email{saindebmalya@gmail.com}

\address[Mal]{Department of Mathematics\\ Jadavpur University\\ Kolkata 700032\\ West Bengal\\ INDIA}
\email{arpitamalju@gmail.com}

\thanks{ The research of the second author is sponsored by Dr. D. S. Kothari Postdoctoral fellowship. The third author would like to thank UGC, Govt. of India for the financial support.} 


\begin{abstract}
There are two notions of approximate Birkhoff-James orthogonality in  a normed space.
We characterize both the notions of approximate Birkhoff-James orthogonality in the space of bounded linear operators defined on a normed space. A  complete characterization of approximate Birkhoff-James orthogonality in the space of bounded linear operators defined on   Hilbert space of any dimension is obtained which improves on the recent result  by Chmieli\'nski et al. [ J. Chmieli\'nski, T. Stypula and P. W\'ojcik, \textit{Approximate orthogonality in normed spaces and its applications}, Linear Algebra and its Applications,  \textbf{531}  (2017), 305--317.], in which they characterized  approximate Birkhoff-James orthogonality of linear operators on finite dimensional Hilbert space and also of compact operators on any Hilbert space.

\end{abstract}

\maketitle

\section{Introduction.} 
There are various notions of orthogonality in a normed space, which are in general different, if the norm is not induced by an inner product. Among all the notions of orthogonality, Birkhoff-James orthogonality\cite{B,J} plays a very important role in the study of geometry of normed spaces. In \cite{J} James elaborated how the notions like smoothness, rotundity etc. of the space can be studied using Birkhoff-James orthogonality. This notion has its importance in studying the geometric properties of the  operator space. Recently Paul et al. \cite{PSG} obtained a sufficient condition for the smoothness of a bounded linear operator using  Birkhoff-James orthogonality of bounded linear operators. Due to the importance of Birkhoff-James orthogonality  it has been generalized by Dragomir\cite{D} and Chmieli\'nski\cite{C} and is known as approximate Birkhoff-James orthogonality. In this paper we characterize approximate Birkhoff-James orthogonality of bounded linear operators which explores the interrelation between that of the  ground space and the space of bounded linear operators. To proceed in details we fix some notations and terminologies.

Let $\mathbb{X}, \mathbb{Y}$ denote real normed spaces and  $\mathbb{H}$, a real Hilbert space.  Let $B_{\mathbb{X}} $ and $ S_{\mathbb{X}} $ denote the unit ball and the unit sphere of $ \mathbb{X} $ respectively, i.e.,  $ B_\mathbb{X}=\{x \in \mathbb{X} : \|x\| \leq 1\} $ and $ S_\mathbb{X}=\{x \in \mathbb{X} : \|x\|=1\}. $ Let $\mathbb{B}(\mathbb{X}, \mathbb{Y})(\mathbb{K}(\mathbb{X}, \mathbb{Y}))$ denote the space of all bounded(compact) linear operators from $\mathbb{X}$ to $\mathbb{Y}$. We write $\mathbb{B}(\mathbb{X}, \mathbb{Y})= \mathbb{B}(\mathbb{X})$ and $\mathbb{K}(\mathbb{X}, \mathbb{Y})= \mathbb{K}(\mathbb{X})$ if $\mathbb{X}= \mathbb{Y}$. A bounded linear operator $T$ is 
 said to attain norm at $x \in S_{\mathbb{X}}$ if $\|Tx\|=\|T\|$. Let $M_T$ denote the set of all unit elements at which $T$ attains norm, i.e., $ M_T=\{x \in S_{\mathbb{X}}:~ \|Tx\|=\|T\|\}.$ The norm attainment set $M_T$ plays an important role in characterizing Birkhoff-James orthogonality of bounded linear operators \cite{PSG,SP,SPH}. \\

For any two elements $x$ and $y$ in $\mathbb{X}$, $x$ is said to be orthogonal to $y$ in the sense of Birkhoff-James, written as $x \perp_B y$, if $\|x + \lambda y \| \geq \|x\|$ for all real scalars $\lambda.$ If the norm of the space is induced by an inner product then $x \perp_B y$ is equivalent to $ \langle x, y \rangle = 0.$ For $T,A \in \mathbb{B}(\mathbb{X}, \mathbb{Y})$, $T$ is said to be orthogonal to $A$ in the sense of Birkhoff-James, written as $T \perp_B A$, if $\|T + \lambda A \| \geq \|T\|$ for all real scalars $\lambda.$ In an inner product space $\mathbb{H}$ approximate orthogonality is defined in the following way: \\
 Let $\epsilon \in [0,1)$. Then for $x,y \in \mathbb{H},$ $x$ is said to be approximate orthogonal to $y$, written as $ x\perp^{\epsilon}y$, if
\[ |\langle x,y \rangle| \leq \epsilon \|x\|\|y\|. \]
Dragomir  \cite{D} introduced the notion of approximate Birkhoff-James orthogonality in a normed space as follows:\\
Let $ \epsilon \in [0,1) .$ Then for $x, y \in \mathbb{X}$, $x$ is said to be approximate $ \epsilon- $ Birkhoff-James orthogonal to $y$  if 
\[ \|x + \lambda y\| \geq (1-\epsilon) \|x\| ~\forall \lambda \in \mathbb{R}.\] 
Later on, Chmieli\'nski \cite{C} slightly modified the definition in the following way: 
Let $ \epsilon \in [0,1).$ Then for $x, y \in \mathbb{X}$, $x$ is said to be approximate $ \epsilon- $ Birkhoff-James orthogonal to $y$  if 
\[ \|x + \lambda y\| \geq \sqrt{1-\epsilon^2} \|x\| ~\forall \lambda \in \mathbb{R} .\]
In this case, we write $ x \bot_D^{\epsilon} y.$ \\

Recently Chmieli\'nski \cite{C}  introduced another notion of approximate Birkhoff-James orthogonality, defined in the following way: \\
Let $ \epsilon \in [0,1) .$ Then for $x, y \in \mathbb{X}$, $x$ is said to be approximate $ \epsilon- $ Birkhoff-James orthogonal to $y$  if 
\[ \|x + \lambda y\|^2 \geq  \|x\|^2 - 2 \epsilon \|x\| \| \lambda y\| ~\forall \lambda \in \mathbb{R} .\] 
In this case, we write $ x \bot_B^{\epsilon} y.$ Chmieli\'nski et al. \cite{CSW} characterized ``$ x \bot_B^{\epsilon} y$'' for real normed spaces as in the following theorem: 
\begin{theorem} [Theorem 2.3,\cite{CSW}]
Let $\mathbb{X}$ be a real normed  space. For $ x, y \in \mathbb{X} $ and $ \epsilon \in [0,1) :$ 
\[ x \bot_B^{\epsilon} y \Leftrightarrow \exists ~z \in ~\mbox{Lin}\{x,y\} : x \bot_B z, \|z-y\| \leq \epsilon \|y\|.\]
\end{theorem} 
It should be noted that in an inner product space, both types of approximate Birkhoff-James orthogonality coincide. However, this is not necessarily true in a normed space. Also note that, in a normed space, both types of approximate Birkhoff-James orthogonality are homogeneous. In this paper we characterize both types of approximate Birkhoff-James orthogonality of bounded linear operators using the norm attainment set. To do so we need the following two definitions introduced in \cite{S} and \cite{SPM} respectively.

\begin{definition}[\cite{S}]
For any two elements $ x, y $ in a real normed  space $ \mathbb{X}, $ let us say that $ y \in x^{+} $ if $ \| x + \lambda y \| \geq \| x \| $ for all $ \lambda \geq 0. $ Accordingly, we say that $ y \in x^{-} $ if $ \| x + \lambda y \| \geq \| x \| $ for all $ \lambda \leq 0$.
\end{definition}
\begin{definition}[\cite{SPM}]
Let $ \mathbb{X} $ be a normed  space and let $ x, y \in \mathbb{X}. $ For $ \epsilon \in [0, 1), $ we say that $ y \in x^{+(\epsilon)} $if $ \| x+\lambda y \|\geq \sqrt[]{1-\epsilon^2}~\|x\|$ for all $\lambda \geq 0. $ Similarly, we say that $ y \in x^{-(\epsilon)} $if $ \| x+\lambda y \|\geq \sqrt[]{1-\epsilon^2}~\|x\|$ for all $\lambda \leq 0. $ 
\end{definition}
 For $ T, A \in \mathbb{K}(\mathbb{X}, \mathbb{Y}) , $ we  characterize $ T \bot^{\epsilon}_D A$ in terms of the norm attainment set $M_T$ when the space $\mathbb{X}$ is a reflexive Banach space. We also provide an alternative proof of Theorem 2.2 of \cite{S} which states that  in a finite dimensional real normed space, $T \perp_B A$ if and only if there exists $x, y \in M_T$ such that $Ax \in (Tx)^+$ and $Ay \in (Ty)^-$. We characterize approximate Birkhoff-James orthogonality $ T \bot^{\epsilon}_B A$ of a bounded linear operator $T$ acting on a Hilbert space of any dimension. This improves on the Theorem 3.4 of \cite{CSW} in which the characterization of approximate Birkhoff-James orthogonality of compact linear operator was obtained. We also provide  characterization of   approximate Birkhoff-James orthogonality $ T \bot^{\epsilon}_B A$ in both the cases of $T,A $ being compact linear operators on a reflexive Banach space and $T,A$ being bounded linear operators on a normed space of any dimension.

\section{ Approximate Birkhoff-James orthogonality  $(\perp_D^{\epsilon})$  of bounded linear operators}

We begin with a complete characterization of approximate Birkhoff-James orthogonality $(\perp_D^{\epsilon})$ of compact linear operators defined on a reflexive Banach space.

\begin{theorem}\label{theorem:operator}
Let $\mathbb{X}$ be a reflexive Banach space and $\mathbb{Y}$ be any normed space. Let $T,A \in \mathbb{K(\mathbb{X}, \mathbb{Y})}$. Then for $0 \leq \epsilon < 1$, $T \perp_D^{\epsilon} A$ if and only if either $(a)$ or $(b)$ holds.\\
$(a)$ There exists $x \in M_T$ such that $Ax\in (Tx)^+$ and for each $\lambda \in (-1-\sqrt{1-{\epsilon}^2}, -1 + \sqrt{1-{\epsilon}^2})\frac{\|T\|}{\|A\|}$ there exists $x_{\lambda} \in S_{\mathbb{X}}$ such that $\|Tx_{\lambda}+ \lambda Ax_{\lambda}\| \geq \sqrt{1-{\epsilon}^2}\|T\|$.\\
$(b)$ There exists $y \in M_T$ such that $Ay\in (Ty)^-$ and for each $\lambda \in (1-\sqrt{1-{\epsilon}^2}, 1 + \sqrt{1-{\epsilon}^2})\frac{\|T\|}{\|A\|}$ there exists $y_{\lambda} \in S_{\mathbb{X}}$ such that $\|Ty_{\lambda}+ \lambda Ay_{\lambda}\| \geq \sqrt{1-{\epsilon}^2}\|T\|$.
\end{theorem}
\begin{proof}
We first prove the necessary condition. Suppose that $T \perp_D^{\epsilon} A$. Since $\mathbb{X}$ is a reflexive Banach space and $T \in \mathbb{K(\mathbb{X}, \mathbb{Y})}$, $M_T \neq \emptyset$. Suppose there exists $x \in M_T$ such that $Ax \in (Tx)^+$. Now since $T,A$ are compact linear operators, $T + \lambda A$ is also a compact linear operator for each $\lambda \in \mathbb{R}$. Therefore, for each $\lambda \in (-1-\sqrt{1-{\epsilon}^2}, -1 + \sqrt{1-{\epsilon}^2})\frac{\|T\|}{\|A\|}$, $(T + \lambda A)$ attains norm at some $x_{\lambda} \in S_{\mathbb{X}}$. Then $\|(T+ \lambda A)x_{\lambda}\|= \|T + \lambda A\| \geq \sqrt{1-{\epsilon}^2}\|T\|$. Thus $(a)$ holds. Similarly suppose there exists $y \in M_T$ such that $Ay \in (Ty)^-$. Now for each $\lambda \in (1-\sqrt{1-{\epsilon}^2}, 1 + \sqrt{1-{\epsilon}^2})\frac{\|T\|}{\|A\|}$, suppose $(T + \lambda A)$ attains norm at $y_{\lambda} \in S_{\mathbb{X}}$. Then $\|(T+ \lambda A)y_{\lambda}\|= \|T + \lambda A\| \geq \sqrt{1-{\epsilon}^2}\|T\|$. Thus $(b)$ holds. This completes the proof of the necessary part.\\
Now we prove the sufficient part. Suppose that $(a)$ holds. Then for all $\lambda \geq 0$, $\|T + \lambda A\| \geq \|(T+ \lambda A)x\|\geq \|Tx\| \geq \sqrt{1-{\epsilon}^2}\|T\|$. For all $\lambda \in (-1-\sqrt{1-{\epsilon}^2}, -1 + \sqrt{1-{\epsilon}^2})\frac{\|T\|}{\|A\|}$, $\|T + \lambda A\|\geq \|(T + \lambda A)x_{\lambda}\|\geq \sqrt{1-{\epsilon}^2}\|T\|$. Now if $(-1 + \sqrt{1-{\epsilon}^2})\frac{\|T\|}{\|A\|} \leq \lambda < 0$, then $\|T+ \lambda A\| \geq \|T\|- |\lambda|\|A\|= \|T\|+ \lambda \|A\| \geq \|T\|- \|T\| + \sqrt{1-{\epsilon}^2}\|T\|= \sqrt{1-{\epsilon}^2}\|T\|$. If $\lambda \leq (-1-\sqrt{1-{\epsilon}^2})\frac{\|T\|}{\|A\|}$, then $\|T + \lambda A\| \geq |\lambda|\|A\|- \|T\| \geq (1+\sqrt{1-{\epsilon}^2})\|T\|- \|T\|= \sqrt{1-{\epsilon}^2}\|T\|$. Therefore, $T \perp_D^{\epsilon}A$. Similarly if $(b)$ holds it can be shown that $T \perp_D^{\epsilon}A$.
\end{proof}
Using the last result we now give another proof of characterization of Birkhoff-James orthogonality [Theorem 2.2, \cite{D}] of bounded linear operators in terms of the norm attainment set.

\begin{theorem}\label{theorem:ortho}
Let $\mathbb{X}$ be a  finite dimensional Banach space. Let $T,A \in \mathbb{B}(\mathbb{X})$. Then $T \perp_B A$ if and only if there exists $x, y \in M_T$ such that $Ax \in (Tx)^+$ and $Ay \in (Ty)^-$. 
\end{theorem}
\begin{proof} The sufficient part of the proof is obvious. We prove the necessary part. 
Being finite dimensional Banach space, the space $\mathbb{X}$ is reflexive and every bounded linear operator defined on $\mathbb{X}$ is a compact linear operator. So we can apply Theorem \ref{theorem:operator}. Without loss of generality we assume that $(a)$ holds. Then there exists $x \in M_T$ such that $Ax \in (Tx)^+$ and for each $\lambda_n = - \frac{1}{n}$ there exists $x_n \in S_{\mathbb{X}}$ such that $\|(T + \lambda_n A)x_n\| \geq \|T\| ,$ i.e, $\|Tx_n - \frac{1}{n} Ax_n\| \geq \|T\|$. Now, since $\mathbb{X}$ is finite dimensional, $\{x_n\}$ has a convergent subsequence say, $\{x_{n_k}\}$ converging to $y$. Therefore, as $n_k \longrightarrow \infty$, we have, $\|Ty\| \geq \|T\| \geq \|Ty\|$. This gives that $y \in M_T$. Now we show that $Ay \in (Ty)^-$. Let $\lambda < - \frac{1}{n} < 0$. Then we claim that $\|Tx_n + \lambda Ax_n\| \geq \|T\|$. Otherwise, $Tx_n - \frac{1}{n} Ax_n = (1-t)(Tx_n + \lambda Ax_n) + t Tx_n$ for some $0 < t < 1$. This implies that $\|Tx_n - \frac{1}{n} Ax_n\| < (1-t)\|T\| + t \|T\|= \|T\|$, a contradiction. This proves our claim. Now for any $\lambda < 0$, there exists $n_0 \in \mathbb{N}$ such that $\lambda < - \frac{1}{n_0} < 0$. Therefore, for all $n \geq n_0$, we have, $\|Tx_n + \lambda Ax_n\| \geq \|T\|$. Letting $n \longrightarrow \infty$, we have, $\|Ty + \lambda Ay\| \geq \|T\|= \|Ty\| $. Hence $Ay \in (Ty)^-$. This completes the proof of the theorem.     
\end{proof}

\section{ Approximate Birkhoff-James orthogonality  $(\perp_B^{\epsilon})$  of bounded linear operators}
In this section we begin with   a complete characterization of  approximate Birkhoff-James orthogonality  $(\perp_B^{\epsilon})$  in the space of bounded linear operators defined on a Hilbert space $\mathbb{H}.$ 

\begin{theorem}\label{theorem:hilbert}
(i) Let $T \in \mathbb{B}(\mathbb{H}).$ Then for any $A\in \mathbb{B}(\mathbb{H})$, $T\bot_B^{\epsilon} A \Leftrightarrow |\langle Tx,Ax\rangle| \leq \epsilon \|T\|\|A\|$ for some $x \in M_T$  if and only if $M_T= S_{H_0}$ for some finite dimensional subspace $H_0$ of $\mathbb{H}$ and $\|T\|_{H_0^{\bot}} < \|T\|$. \\
(ii) Moreover, if $M_T \subseteq M_A$ then $T\bot_B^{\epsilon} A \Leftrightarrow Tx \bot^{\epsilon} Ax$ for some $x \in M_T$  if and only if $M_T= S_{H_0}$ for some finite dimensional subspace $H_0$ of $\mathbb{H}$ and $\|T\|_{H_0^{\bot}} < \|T\|$. 
\end{theorem}

\begin{proof}  Without loss of generality we assume that $\|T\| = 1.$ \\
(i) We first prove the necessary part. From Theorem 2.2 of Sain and Paul \cite{SP} it follows that in case of a Hilbert space the norm attaining set $M_T$ is always a unit sphere of some subspace of the space. We  first show that the subspace is finite dimensional. Suppose $M_T$ be the unit sphere of an infinite dimensional subspace $H_0.$ Then we can find a set $\{e_n : n \in \mathbb{N} \}$ of orthonormal   vectors in $H_0.$  Extend the set to a complete orthonormal basis $\mathcal{B} = \{ e_{\alpha} : \alpha \in \Lambda \supset \mathbb{N} \} $ of $\mathbb{H}.$
For each $e_\alpha \in H_0 \cap \mathcal{B} $ we have 
$$ \| T^{*}T \| = \|T\|^2 = \|Te_\alpha \|^2 = \langle T^{*}T e_\alpha , e_\alpha \rangle \leq \|T^{*}Te_\alpha \| \|e_\alpha \| \leq \|T^{*}T\| $$
so that by the equality condition of Schwarz's inequality we get $ T^{*}Te_\alpha = \lambda_\alpha e_\alpha $ for some scalar $\lambda_\alpha.$ 
Thus $\{ Te_\alpha : e_\alpha \in H_0 \cap \mathcal{B}\}$ is a set of orthonormal vectors in $\mathbb{H}.$ For $ \epsilon \in [0,1) $ there exists $ c_n $ such that $ \epsilon < c_n < 1 $ and $ c_n \longrightarrow \epsilon.$
 Define $A : \mathcal{B} \longrightarrow \mathbb{H} $ as follows : 
\begin{eqnarray*}
           A(e_n) & =   & c_n Te_n, n \in \mathbb{N} \\
         	 A(e_\alpha) & = & Te_\alpha,~ e_\alpha \in  H_0 \cap \mathcal{B} - \{e_n : n \in \mathbb{N} \}\\
					 A(e_\alpha) & = & 0,~e_\alpha \in \mathcal{B} - H_0 \cap \mathcal{B}
\end{eqnarray*}	

As $\{ Te_\alpha : e_\alpha \in H_0 \cap \mathcal{B}\}$ is a set of orthonormal vectors in $\mathbb{H}$ it is easy to see that  $A$ can be extended as  a bounded linear operator on $\mathbb{H}$ and $ \|A\| = 1.$ 

Now for  any scalar $\lambda, ~ \| T + \lambda A \|^2 \geq \| (T + \lambda A)e_n \|^2 = \| ( 1 + \lambda c_n)Te_n \|^2 = 1 + \lambda^2 {c_n}^2 + 2 \lambda c_n \longrightarrow 1 + \lambda^2 \epsilon^2 + 2 \lambda \epsilon.$ It is easy to see that $ \| T + \lambda A\|^2 \geq \|T\|^2 - 2 \epsilon \|T\| \|\lambda A\|~ \forall \lambda$ and hence $ T \bot_{B}^\epsilon A.$  

We next show that there exists no $ x \in M_T$ such that $ \mid \langle Tx, Ax \rangle \mid  \leq  \epsilon \|T\|\|A\| .$
Let $ x = \sum_{\alpha} \langle x, e_\alpha \rangle e_\alpha~ \in M_T.$  Then 
$$ \langle Tx, Ax \rangle = \sum_n c_n \mid \langle x, e_n \rangle\mid ^2	\|Te_n\|^2+ \sum_{\alpha \notin \mathbb{N}} \mid \langle x, e_\alpha \rangle\mid ^2 	\|Te_\alpha\|^2 > \epsilon \|x\|^2$$
and so $ \mid \langle Tx, Ax \rangle \mid > \epsilon \|T\|\|A\| $ for each $x \in M_T.$
Thus $T \bot_B^\epsilon A$ but there exists no $x \in M_T$ such that $ \mid \langle Tx, Ax \rangle \mid \leq \epsilon \|T\|\|A\| .$ This is a contradiction and so $H_0$ must be finite dimensional.

Next we claim that $\|T\|_{{H_0}^\bot} < \|T\|$. Suppose that $\|T\|_{{H_0}^\bot} = \|T\|$. Since $T$ does not attain norm on $H_0^\bot$, there exists a sequence $ \{e_n \} $ in $H_0^\bot$ such that $\|Te_n\| \rightarrow \|T\|$. We have $ \mathbb{H} = H_0 \oplus {H_0}^{\bot} $.

Define $A : \mathbb{H} \longrightarrow \mathbb{H} $ as follows:
\[ Az = Tx, \mbox{ where } z = x + y, x \in H_0, y \in {H_0}^{\bot} \]
 Then it is easy to check $A$ is bounded on $\mathbb{H}$ and $\|A\|=1$. Also for any scalar $ \lambda $, $ \| T + \lambda A \|^2 \geq \| (T + \lambda A) e_n \|^2  = \| Te_n \|^2 $ holds for each $ n \in \mathbb{N}.$ Thus $ \| T + \lambda A \|^2 \geq \|T\|^2\geq \|T\|^2- 2 \epsilon \|T\|\|\lambda A\| $ for all $ \lambda $. Hence $ T \bot_B^{\epsilon} A.$ 
But for all $ x\in M_T$, $|\langle Tx, Ax\rangle |\geq \langle Tx, Ax\rangle = \langle Tx, Tx \rangle =\|Tx\|^2= 1 > \epsilon \|T\|\|A\| .$ This contradiction completes the proof of the  necessary part of the theorem.\\

  We next prove the sufficient part. Let $T\bot_B^{\epsilon} A$. Then by Theorem 2.3 of \cite{CSW} there exists $S\in \mathbb{B}({\mathbb{H}})$ such that $T\bot_B S$ and $\|A-S\|\leq \epsilon \|A\|$. Since $M_T= S_{H_0},$ where $H_0$ is a finite dimensional subspace of $\mathbb{H}$ and $ \| T\|_{H_0^{\bot}} < \|T\|, $ so by Theorem 3.1 of \cite{PSG} there exists $x \in M_T$ such that $Tx \bot_B Sx$. Clearly $\|Sx-Ax\|\leq \|S-A\| \leq \epsilon \|A\|$. Now $\langle Tx, Ax \rangle = \langle Tx, Sx \rangle + \langle Tx, Ax- Sx \rangle \Rightarrow |\langle Tx, Ax \rangle | \leq |\langle Tx, Sx \rangle | + |\langle Tx, Ax- Sx \rangle | \leq \|Tx\|\|Ax-Sx\| \leq \epsilon \|T\|\|A\|$.\\
Conversely, let $|\langle Tx, Ax \rangle| \leq \epsilon \|T\|\|A\|$ for some $x\in M_T$. Then for any $\lambda \in \mathbb{R}$, $\|T+\lambda A\|^2 \geq \|(T+\lambda A)x\|^2= \|Tx\|^2+ 2\lambda \langle Tx, Ax\rangle + |\lambda|^2 \|Ax\|^2 \geq \|Tx\|^2 +2 \lambda \langle Tx, Ax \rangle \geq \|Tx\|^2 -2 |\lambda| |\langle Tx, Ax \rangle| \geq \|T\|^2 -2\epsilon |\lambda| \|T\|\|A\|$. Hence $T \bot_B^{\epsilon} A$.
This completes the proof of (i). \\
(ii) Since $M_T \subseteq M_A$ so $x \in M_T$ implies $ \|Tx\| = \|T\|, \|Ax\| = \|A\|$ and hence (ii) follows from (i).    
\end{proof}
The above theorem improves on both Theorem 3.3 [Part (3)] and Theorem 3.4 of  Chmieli\'nski et al. \cite{CSW}.   
If the Hilbert space $\mathbb{H}$ is finite dimensional  then we have the following corollary, the proof of which is obvious.
\begin{cor}[Theorem 3.2 \cite{CSW}]\label{cor:finite}
 Let $\mathbb{H}$ be a finite dimensional Hilbert space and  $T \in \mathbb{B}(\mathbb{H})$. Then for any $A\in \mathbb{B}(\mathbb{H})$, $T\bot_B^{\epsilon} A \Leftrightarrow |\langle Tx,Ax\rangle| \leq \epsilon \|T\|\|A\|$ for some $x \in M_T.$\\
Moreover, if $M_T \subseteq M_A$ then $T\bot_B^{\epsilon} A \Leftrightarrow Tx \bot^{\epsilon} Ax$ for some $x \in M_T.$ 
\end{cor}

If $T$ is a compact linear operator on $\mathbb{H},$ not necessarily finite dimensional, then we have 
the following corollary. 
\begin{cor}[Theorem 3.4 \cite{CSW}]\label{cor:compact}
Let  $\mathbb{H}$ be a Hilbert space and $T \in \mathbb{K}(\mathbb{H}).$ If $A \in \mathbb{B}(\mathbb{H})$ and $M_T \subseteq M_A$  then $T\perp_B^{\epsilon} A$ if and only if there exists $x \in M_T$ such that $Tx\perp^{\epsilon} Ax$. 
\end{cor}
\begin{proof}
Since $\mathbb{H}$ is a Hilbert space, we have from Theorem 2.2 of Sain and Paul \cite{SP}, $M_T= S_{H_0}$ for some subspace $H_0$ of $\mathbb{H}$. Now it is easy to observe that for each $x \in H_0 $, $T^*T x= \|T\|^2 x$. Since $T^*T$ is also a compact operator, so $H_0$ must be a finite dimensional subspace of $\mathbb{H}$. Since $H_0^{\bot}$ is reflexive and $T \in \mathbb{K}(\mathbb{H})$, clearly $\|T\|_{H_0^{\bot}} < \|T\|$. Hence by applying part (ii) of Theorem \ref{theorem:hilbert} we have, $T\perp_B^{\epsilon} A$ if and only if there exists $x \in M_T$ such that $Tx\perp^{\epsilon} Ax$. 
\end{proof}
\begin{remark}
Let $\{e_n\}$ be the standard orthonormal basis of the Hilbert space $\ell^2.$ 
Consider $ T : \ell^2 \longrightarrow \ell^2$ defined by
\begin{eqnarray*}
Te_1 & = & e_1\\
Te_n & = & \Big(\frac{1}{2} - \frac{1}{n+2}\Big)e_n,~ n >1 
\end{eqnarray*}
Then T does not satisfy the conditions of Corollary \ref{cor:finite} and \ref{cor:compact} but $T$ satisfies the condition of Theorem \ref{theorem:hilbert}. This shows that Theorem \ref{theorem:hilbert} includes a larger class of operators.
\end{remark}

Next we study approximate Birkhoff-James orthogonality  $(\perp_B^{\epsilon})$  of bounded linear operators on normed spaces. If the space $\mathbb{X}$ is reflexive and $T,A \in \mathbb{K}(\mathbb{X}, \mathbb{Y})$ then we
obtain a necessary and sufficient condition for $T \perp_B^{\epsilon} A$, provided $M_T=D\cup (-D)$ ($D$ being a nonempty compact connected subset of $S_{\mathbb{X}}$).

\begin{theorem}\label{theorem:compact}
Let $\mathbb{X}$ be a reflexive Banach space and $\mathbb{Y}$ be a normed  space. Let $T,A \in \mathbb{K}(\mathbb X, \mathbb Y)$ and $M_T=D\cup (-D),$ where $D$ is a nonempty compact connected subset of $S_{\mathbb{X}}$. Then $T\bot_B^{\epsilon} A$ if and only if there exists $x\in M_T$ such that $\|Tx+ \lambda Ax\|^2 \geq \|T\|^2 -2\epsilon \|T\|\|\lambda A\|$.\\
 Moreover if $M_T\subseteq M_A$, then $T\bot_B^{\epsilon} A$ if and only if $Tx \bot_B^{\epsilon} Ax$. 
\end{theorem}
\begin{proof}
We first prove the necessary part. \\
Since $T\bot_B^{\epsilon} A$, by Theorem 2.3 in \cite{CSW} there exists a linear operator $S\in Span \{T,~ A\}$ such that $T\bot_B S$ and $\|S-A\|\leq \epsilon \|A\|$. Since $T, A$ are compact and $S\in Span \{T, A\}$, $S$ is also compact. Therefore, by Theorem 2.1 in \cite{PSG} there exists $x\in M_T$ such that $Tx \bot_B Sx$. Clearly $\|Sx-Ax\|\leq \|S-A\|\leq \epsilon \|A\|$. Thus for any $\lambda \in \mathbb{R}$, $\|Tx + \lambda Ax\|^2 = \|(Tx+ \lambda Sx) - \lambda( Sx - Ax)\|^2 \geq \|Tx + \lambda Sx\|^2 - 2|\lambda| \|Tx+ \lambda Sx\| \|Sx - Ax\|\geq \|Tx\|^2 - 2 \epsilon|\lambda|\|Tx\|\|A\|=\|T\|^2- 2\epsilon\|T\|\|\lambda A\|$. \\
For the sufficient part, let $\|Tx+ \lambda Ax\|^2 \geq \|T\|^2 -2\epsilon \|T\|\|\lambda A\|$ for some $x\in M_T$. Then for any $\lambda \in \mathbb{R}$, $\|T+\lambda A\|^2\geq \|Tx + \lambda Ax\|^2 \geq \|T\|^2- 2\epsilon\|T\|\|\lambda A\|$. Hence $T\bot_B^{\epsilon} A$.\\
Moreover, let $M_T\subseteq M_A$. Then $T\bot_B^{\epsilon} A \Rightarrow \|Tx+ \lambda Ax\|^2 \geq \|T\|^2 -2\epsilon \|T\|\|\lambda A\|$ for some $x\in M_T$. This implies that $\|Tx+ \lambda Ax\|^2 \geq \|Tx\|^2 -2\epsilon \|Tx\|\|\lambda Ax\| \Rightarrow Tx \bot_B^{\epsilon} Ax$. \\
Conversely, let for some $x\in M_T\subseteq M_A$, $Tx \bot_B^{\epsilon} Ax$. Then for any $\lambda \in \mathbb{R}$, $\|T+ \lambda A\|^2 \geq \|Tx+ \lambda Ax\|^2 \geq \|Tx\|^2 - 2 \epsilon \|Tx\|\|\lambda Ax\|= \|T\|^2- 2 \epsilon \|T\|\|\lambda A\|$. Thus $T\bot_B^{\epsilon} A$.
\end{proof}

In general the norm attainment set of a compact operator may not be of the above form $ D \cup (- D).$ In the next theorem we give a complete characterization of approximate Birkhoff-James orthogonality $(\perp_B^{\epsilon})$ of compact linear operators without any restriction on the norm attainment set.

\begin{theorem}\label{theorem:compact complete}
Let $\mathbb{X}$ be a reflexive Banach space and $\mathbb{Y}$ be a normed  space. Let $T,A \in \mathbb{K}(\mathbb X, \mathbb Y)$. Then $T\bot_B^{\epsilon} A$ if and only if there exists $x,y \in M_T$ such that $\|Tx+ \lambda Ax\|^2 \geq \|T\|^2 -2\epsilon \|T\|\|\lambda A\|$ for all $\lambda \geq 0$ and $\|Ty+ \lambda Ay\|^2 \geq \|T\|^2 -2\epsilon \|T\|\|\lambda A\|$ for all $\lambda \leq 0$. 
\end{theorem}
\begin{proof}
We first prove the necessary part. \\
Since $T\bot_B^{\epsilon} A$, by Theorem 2.3 in \cite{CSW} there exists a linear operator $S \in span \{T,~ A\}$ such that $T\bot_B S$ and $\|S-A\|\leq \epsilon \|A\|$. Since $T, A$ are compact and $S\in Span \{T,A\}$, $S$ is also compact. Therefore, by Theorem 2.1 in \cite{SPM} there exists $x,y\in M_T$ such that $Sx\in (Tx)^+$ and $Sy\in (Ty)^-$. Clearly $\|Sx-Ax\|\leq \|S-A\|\leq \epsilon \|A\|$. Thus for any $\lambda \geq 0$, $\|Tx + \lambda Ax\|^2 = \|(Tx+ \lambda Sx) - \lambda( Sx - Ax)\|^2 \geq \|Tx + \lambda Sx\|^2 - 2|\lambda| \|Tx+ \lambda Sx\| \|Sx - Ax\|\geq \|Tx\|^2 - 2 \epsilon|\lambda|\|Tx\|\|A\|=\|T\|^2 - 2 \epsilon|\lambda|\|T\|\|A\| $. Similarly using $Sy \in (Ty)^-$ and $\|S-A\|\leq \epsilon \|A\|$, it can be shown that $\|Ty+ \lambda Ay\|^2 \geq \|T\|^2 -2\epsilon \|T\|\|\lambda A\|$ for all $\lambda \leq 0$.\\
For the sufficient part, suppose that there exists $x,y \in M_T$ such that $\|Tx+ \lambda Ax\|^2 \geq \|T\|^2 -2\epsilon \|T\|\|\lambda A\|$ for all $\lambda \geq 0$ and $\|Ty+ \lambda Ay\|^2 \geq \|T\|^2 -2\epsilon \|T\|\|\lambda A\|$ for all $\lambda \leq 0$. Then for any $\lambda \geq 0$, $\|T+\lambda A\|^2 \geq \|Tx + \lambda Ax\|^2 \geq \|T\|^2 - 2 \epsilon \|T\|\|\lambda A\|$. Similarly, for any $\lambda \leq 0$, $\|T+\lambda A\|^2 \geq \| Ty + \lambda Ay \|^2 \geq \|T\|^2 - 2 \epsilon \|T\|\|\lambda A\|$. Therefore, $T\bot_B^{\epsilon} A$.
\end{proof}

For bounded linear operator defined on any normed space the situation is far more complicated since the norm attainment set may be empty in this case. In the next theorem we characterize approximate Birkhoff-James orthogonality $(\perp_B^{\epsilon})$ of bounded linear operators defined on any normed space.

\begin{theorem}\label{theorem:bounded}
Let $\mathbb{X}, \mathbb{Y}$ be two normed spaces. Suppose $T\in \mathbb{B}(\mathbb{X}, \mathbb{Y})$ be nonzero. Then for any $A\in \mathbb{B}(\mathbb{X}, \mathbb{Y})$, $T\bot_B^{\epsilon} A$ if and only if either of the conditions in $(a)$ or $(b)$ holds.\\
$(a)$ There exists a sequence $\{x_n\}$ of unit vectors such that $\|Tx_n\|\longrightarrow \|T\|$ and $lim_{n\rightarrow \infty} \|Ax_n\| \leq \epsilon \|A\|$.\\
$(b)$ There exists two sequences $\{x_n\}, \{y_n\}$ of unit vectors and two sequences of positive real numbers $\{\epsilon_n\}, \{\delta_n\}$ such that\\
$(i)~\epsilon_n \longrightarrow 0,~ \delta_n \longrightarrow 0,~ \|Tx_n\|\longrightarrow \|T\|,~\|Ty_n\|\longrightarrow \|T\|$ as $n \longrightarrow \infty$.\\
$(ii)~ \|Tx_n + \lambda Ax_n\|^2 \geq (1-{\epsilon_n}^2)\|Tx_n\|^2- 2 \epsilon \sqrt{1-{\epsilon_n}^2}\|Tx_n\|\|\lambda A\|$ for all $\lambda \geq 0$.\\
$(iii)~ \|Ty_n + \lambda Ay_n\|^2 \geq (1-{\delta_n}^2)\|Ty_n\|^2- 2 \epsilon \sqrt{1-{\delta_n}^2}\|Ty_n\|\|\lambda A\|$ for all $\lambda \leq 0$.
\end{theorem}
\begin{proof}
We first prove the necessary part. Since $T\bot_B^{\epsilon} A$ so by Theorem 2.3 in \cite{CSW} there exists $S \in span\{T, A\}$ such that $T\bot_B S$ and $\|S- A\|\leq \epsilon \|A\|$. Since $T \bot_B S$ so by Theorem 2.4   in \cite{SPM} either $(i)$ or $(ii)$ holds.\\
$(i)$ There exists a sequence $\{x_n\}$ of unit vectors such that $\|Tx_n\|\longrightarrow \|T\|$ and $\|Sx_n\| \longrightarrow 0$.\\
$(ii)$ There exists two sequences $\{x_n\}, \{y_n\}$ of unit vectors and two sequences of positive real numbers $\{\epsilon_n\}, \{\delta_n\}$ such that $\epsilon_n \longrightarrow 0,~ \delta_n \longrightarrow 0,~ \|Tx_n\|\longrightarrow \|T\|,~\|Ty_n\|\longrightarrow \|T\|$ as $n \longrightarrow \infty$ and $Sx_n \in (Tx_n)^{(+ \epsilon_n)}, $ $ Sy_n \in (Ty_n)^{(-\delta_n)}$ for all $n \in \mathbb{N}$.\\
First suppose that $(i)$ holds. Then $\|Sx_n- Ax_n\| \leq \|S-A\| \leq \epsilon \|A\|$. This implies that $lim_{n\rightarrow \infty} \|Ax_n\| \leq \epsilon \|A\|$.\\
Now suppose that $(ii)$ holds. Then for all $\lambda \geq 0$ we have $\|Tx_n+ \lambda Ax_n\|^2= \|(Tx_n + \lambda Sx_n) + \lambda (Ax_n-Sx_n)\|^2 \geq |\|Tx_n+ \lambda Sx_n\|- |\lambda|\|Ax_n- Sx_n\||^2 \geq \|Tx_n+ \lambda Sx_n\|^2 -2 \|Tx_n + \lambda Sx_n\||\lambda|\|Sx_n- Ax_n\| \geq (1- {\epsilon_n}^2)\|Tx_n\|^2 - 2 \epsilon \sqrt{1-{\epsilon_n}^2}\|Tx_n\|\|\lambda A\|$. Similarly using the sequence $\{y_n\}$ we obtain condition $(iii)$. This proves the necessary part of the theorem. \\
Now for the sufficient part let $(a)$ holds. Then for any $\lambda \in \mathbb{R}$ we have $\|T+ \lambda A\|^2 \geq \|(T+\lambda A)x_n\|^2 \geq |\|Tx_n\|- |\lambda|\|Ax_n\||^2 \geq \|Tx_n\|^2-2 \|Tx_n\||\lambda|\|Ax_n\|$. Letting $n\longrightarrow \infty$ we have $\|T+ \lambda A\|^2 \geq \|T\|^2 -2 \epsilon \|T\|\|\lambda A\|$. Hence $T \bot_B^{\epsilon} A$. Now suppose conditions in $(b)$ hold. Then for any $\lambda \geq 0$ we have $\|T+ \lambda A\|^2 \geq \|(T+ \lambda A)x_n\|^2 \geq (1-{\epsilon_n}^2)\|Tx_n\|^2- 2 \epsilon \sqrt{1-{\epsilon_n}^2}\|Tx_n\|\|\lambda A\|$. Letting $n \longrightarrow \infty$ we have $\|T+ \lambda A\|^2 \geq \|T\|^2 -2 \epsilon \|T\|\|\lambda A\|$. Similarly for any $\lambda \leq 0$ we have $\|T+ \lambda A\|^2 \geq \|(T+ \lambda A)y_n\|^2 \geq (1-{\delta_n}^2)\|Ty_n\|^2- 2 \epsilon \sqrt{1-{\delta_n}^2}\|Ty_n\|\|\lambda A\|$. Letting $n \longrightarrow \infty$ we have $\|T+ \lambda A\|^2 \geq \|T\|^2 -2 \epsilon \|T\|\|\lambda A\|$. Hence $T \bot_B^{\epsilon} A$. 
\end{proof}


\begin{thebibliography}{99}

\bibitem{B} G. Birkhoff,
  \textit{Orthogonality in linear metric spaces},
  Duke Math. J., \textbf{1} (1935) 169--172.

\bibitem{C} J. Chmieli\'nski,
      \textit{On an $\epsilon-$Birkhoff orthogonality},
      Journal of Inequalities in Pure and Applied Mathematics,
  \textbf{6}(3) (2005) Article 79.	
	
	\bibitem{CSW} J. Chmieli\'nski, T. Stypula and P. W\'ojcik, \textit{Approximate orthogonality in normed spaces and its applications}, Linear Algebra and its Applications,  \textbf{531}  (2017), 305--317.

\bibitem{D} S.S. Dragomir, \textit{On approximation of continuous linear functionals in normed linear spaces}, An.
Univ. Timi¸soara Ser. ¸Stiin¸t. Mat., \textbf{29} (1991), 51--58.
		
	\bibitem{J} R. C. James, 
	  \textit{Orthogonality and linear functionals in normed linear spaces},  
		Transactions of the American Mathematical Society, \textbf{61} (1947) 265-292.		
	
	\bibitem{PSG} K. Paul, D. Sain and P. Ghosh,
  \textit{ Birkhoff-James orthogonality and smoothness of bounded linear operators},
  Linear Algebra and its Applications,
  \textbf{506} (2016) 551-563.
	
	
	\bibitem{S} D. Sain,
  \textit{ Birkhoff-James orthogonality of linear operators on finite dimensional Banach spaces},
  Journal of Mathematical Analysis and Applications,
  \textbf{447} (2017) 860-866.

	
		
\bibitem{SP} D. Sain  and K. Paul,
  \textit{ Operator norm attainment and inner product spaces}, 
  Linear Algebra and its Applications,
  \textbf{439} (2013) 2448-2452.
	
	
	\bibitem{SPH} D. Sain, K. Paul and S. Hait,
  \textit{ Operator norm attainment and Birkhoff-James orthogonality},
  Linear Algebra and its Applications,
  \textbf{476} (2015) 85-97.
	
	
	
	\bibitem{SPM} D.Sain, K. Paul, A. Mal, 
  \textit{ A complete characterization of Birkhoff-James orthogonality of bounded linear operators}, 	arXiv:1706.07713 [math.FA]
	
	
\end{thebibliography}
\end{document}